\documentclass[a4paper,UKenglish,cleveref, autoref, thm-restate]{lipics-v2021}
\usepackage{amsthm}
\usepackage{amssymb}
\usepackage{amsmath}
\usepackage{hyperref}
\usepackage{xcolor}
\usepackage{enumerate}
\usepackage{listings}
\usepackage{microtype}
\usepackage{stackengine}
\usepackage{comment}
\usepackage{array}
\usepackage{todonotes}

\newtheorem{problem}{Problem}

\pdfoutput=1 
\hideLIPIcs  

\title{Crossing and non-crossing families} %

\author{Todor {Anti\'c}}{Department of Applied Mathematics, Faculty of Mathematics and Physics, Charles University, Czech Republic}{todor@kam.mff.cuni.cz}{https://orcid.org/0009-0008-6521-7987}{supported by grant no. 23-04949X of the Czech Science Foundation (GA\v{C}R).}

\author{Martin Balko}{Department of Applied Mathematics, Faculty of Mathematics and Physics, Charles University, Czech Republic}{balko@kam.mff.cuni.cz}{https://orcid.org/0000-0001-9688-9489}{supported by grant no. 23-04949X of the Czech Science Foundation (GA\v{C}R) and by the Center for Foundations of Modern Computer Science (Charles Univ. project UNCE 24/SCI/008).}

\author{Birgit Vogtenhuber}{Institute of Algorithms and Theory, Graz University of Technology, Austria}{birgit.vogtenhuber@tugraz.at}{https://orcid.org/0000-0002-7166-4467}{}

\authorrunning{T. Anti\'c, M. Balko, and B. Vogtenhuber} 

\Copyright{Todor Anti\'c, Martin Balko, and Birgit Vogtenhuber} 

\ccsdesc[500]{Theory of computation~Randomness, geometry and discrete structures}
\ccsdesc[500]{Theory of computation~Computational geometry}
\ccsdesc[500]{Mathematics of computing~Combinatorics} 

\keywords{crossing family, non-crossing family, geometric graph} 

\category{Track 1} 

\relatedversion{} 

\acknowledgements{We would like to thank the organizers of the the 19th European Research Week on Geometric Graphs (GGWeek 2024) where this research was initiated.
 We would also like to thank Joachim Orthaber, Bettina Speckmann, and Viola M\'{e}sz\'{a}ros for interesting discussions about the problem during the early stages of the research. We also thank the anonymous reviewers for their valuable comments and suggestions, which helped improve the quality of this work. }

\nolinenumbers 

\EventEditors{John Q. Open and Joan R. Access}
\EventNoEds{2}
\EventLongTitle{42nd Conference on Very Important Topics (CVIT 2016)}
\EventShortTitle{CVIT 2016}
\EventAcronym{CVIT}
\EventYear{2016}
\EventDate{December 24--27, 2016}
\EventLocation{Little Whinging, United Kingdom}
\EventLogo{}
\SeriesVolume{42}
\ArticleNo{23}

\begin{document}

\maketitle

\begin{abstract}
For a finite set $P$ of points in the plane in general position, a \emph{crossing family} of size $k$ in $P$ is a collection of $k$ line segments with endpoints in $P$ that are pairwise crossing.
It is a long-standing open problem to determine the largest size of a crossing family in any set of $n$ points in the plane in general position.
It is widely believed that this size should be linear in $n$.

Motivated by results from the theory of partitioning complete geometric graphs, we study a variant of this problem for point sets $P$ that do not contain a \emph{non-crossing family} of size $m$, which is a collection of 4 disjoint subsets $P_1$, $P_2$, $P_3$, and $P_4$ of $P$, each containing $m$ points of $P$, such that for every choice of 4 points $p_i \in P_i$, the set $\{p_1,p_2,p_3,p_4\}$ is such that  $p_4$ is in the interior of the triangle formed by $p_1,p_2,p_3$. We prove that, for every $m \in \mathbb{N}$, each set $P$ of $n$ points in the plane in general position contains either a crossing family of size $n/2^{O(\sqrt{\log{m}})}$ or a non-crossing family of size $m$, by this strengthening a recent breakthrough result by Pach, Rubin, and Tardos~(2021).
Our proof is constructive and we show that these families can be obtained in expected time $O(nm^{1+o(1)})$.
We also prove that a crossing family of size $\Omega(n/m)$ or a non-crossing family of size $m$ in~$P$ can be found in expected time $O(n)$.
\end{abstract}

\section{Introduction}

Let $P$ be a finite set of points in the plane.
We say that $P$ is \emph{in general position} if there is no line containing three points of $P$.
Two line segments in the plane are \emph{crossing} if their intersection is a single point in the relative interior of both of them.
If $\mathcal{F}$ is a set of $k$ line segments in the plane, each having both endpoints in~$P$, then we say that $\mathcal{F}$ is a \emph{crossing family of size $k$} in $P$ if any two line segments in $\mathcal{F}$ cross; 
see \cref{fig-example}(a) for an example.

A classic topic in discrete geometry is to find the largest size $T(n)$ of a crossing family in any set of $n$ points in the plane in general position.
The study of this problem was initiated in~1991 by Aronov, Erd\H{o}s, Goddard, Kleitman, Klugerman, Pach, and Schulman~\cite{aegkkps94}, who proved that $T(n) \in \Omega(\sqrt{n})$.
On the other hand, it is trivial to see that $T(n) \leq \lfloor n/2 \rfloor$. 
The currently best known upper bound is $T(n) \leq \lceil 8n/41 \rceil$, obtained by Aichholzer, Kyn{\v c}l, Scheucher, Vogtenhuber, and Valtr~\cite{aiKyScgVoVa22}.
There is a prevailing conjecture that $T(n) \in \Theta(n)$~\cite{aegkkps94,braMosPach05}.

\begin{problem}[\cite{braMosPach05}]\label{prob-crossingfamily}
	Is there a constant $c > 0$ such that every set of $n$ points in the plane in general position contains a crossing family of size at least $cn$? 
\end{problem}

Despite many attempts and several studied variants of the problem---see for example~\cite{aiKyScgVoVa22,alPachPinRado05,fulSuk13,laraRub19,prt21,valtr96}---to date no one has been able to resolve this problem. 
In fact, for a long time, not even an improvement of the lower bound of $\Omega(\sqrt{n})$ was obtained.
In 2021, Pach, Rubin, and Tardos~\cite{prt21} achieved a breakthrough result
by proving that $T(n) \in n^{1-o(1)}$.
More precisely, they showed the following result.

\begin{theorem}[\cite{prt21}]
\label{thm-natan}
Every set of $n$ points in the plane in general position contains a crossing family of size at least $n/2^{O(\sqrt{\log{n}})}$.
\end{theorem}

Despite this remarkable progress, the problem of deciding whether $T(n) \in \Theta(n)$ remains open.
Note that the problem becomes trivial if restricted to 
point sets in \emph{convex position}, that is, point sets whose points are vertices of a convex polygon.
It is easy to see that any set of $n$ points in convex position determines a crossing family of size $\lfloor n/2 \rfloor$.

In this work, we study a variant of the crossing family problem where we relate the size of a largest crossing family in a point set to how convex/non-convex the point set is.
A \emph{non‐crossing family} in a set $P$ of points in the plane is a collection of four disjoint non-empty subsets $P_1$, $P_2$, $P_3$, and $P_4$ of $P$, 
such that for every choice of four points $p_i \in P_i$ the set $\{p_1,p_2,p_3,p_4\}$ is not in convex position, with~$p_4$ in the interior of the convex hull of $\{p_1,p_2,p_3\}$.
The \emph{size} of a non-crossing family is the minimum over the cardinalities of the four subsets $P_1$, $P_2$, $P_3$, and $P_4$.
see \cref{fig-example}(b) for an example.

\begin{figure}[ht]
    \centering
    \includegraphics{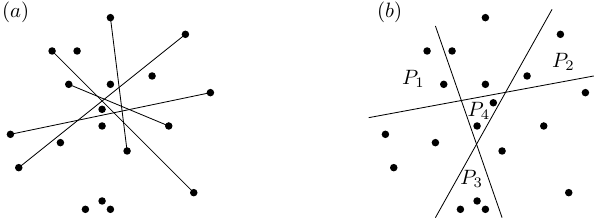}
    \caption{(a) A crossing family of size 5. (b) A non-crossing family of size 2.}
    \label{fig-example}
\end{figure}

It is easy to see and follows from Carath\'{e}dory's theorem 
that
a point set is in convex position if and only if it does not contain a non-crossing family of size one. 
Then, as we have mentioned, it is easy to find a linear-size crossing family.
Towards answering the question of whether $T(n) \in \Theta (n)$, we parameterize point sets according to the largest non-crossing family they contain, and try to find their largest crossing family in dependence of this parameter.
We thus study the following Ramsey-type problem.

\begin{problem}
\label{problem-crVsNoncr}
Given positive integers $m$ and $n$, what is the size of a largest crossing family in any set of $n$ points in the plane in general position that does not contain a non-crossing family of size $m$? 
\end{problem}

In particular, for which values of $m$ can we always find a crossing family that is larger than the bound
$n/2^{O(\sqrt{\log{n}})}$ from \cref{thm-natan} for any set of $n$ points in general position?

\subsection{Motivation: partitioning of complete geometric graphs}
The motivation for studying Problem~\ref{problem-crVsNoncr} also comes from the theory of partitioning complete geometric graphs into subgraphs.
A \emph{complete geometric graph} is a pair $(V,E)$, where $V$ is a set of points in the plane in general position and $E$ is the set of all closed line segments with endpoints in $V$.
The elements of $V$ are called \emph{vertices} and the elements of $E$ are called \emph{edges}.

In 2006, Bose, Hurtado, Rivera-Campo, and Wood~\cite{bhrw06} introduced the problem of partitioning the edge set of a complete geometric graph into plane trees, or, more generally, into plane subgraphs.
Recently, these problems have attracted considerable attention~\cite{aoorsstv22,ahmAlf20,dumiPach24GD,pss23}.

It is not hard to see that the edge set of every complete geometric graph on $n$ vertices can be partitioned into $n-1$ plane subgraphs.
The best known lower bound $\frac{n}{2}+1$~\cite{obeOrt21} is also linear in $n$, 
but it is not known whether a partition into 
$cn$ plane subgraphs for some constant $c<1$ is always possible.
However, if every point set contains a large crossing family or a large non-crossing family, we get such an upper bound: 
Bose, Hurtado, Rivera-Campo, and Wood~\cite{bhrw06} proved that the edge set of every complete geometric graph on $n$ vertices 
with $k$ pairwise crossing edges 
can be partitioned into $n - k$ plane trees (which are not necessarily spanning).
By Theorem~\ref{thm-natan}, this gives an upper bound of $n-n/2^{O(\sqrt{\log{n}})}$ for the partition problem.
On the other hand, Pach, Saghafian, and Schnider~\cite{pss23} showed that the edge set of every complete geometric graph on $n$ vertices whose vertex set contains a non-crossing family of size $bn$ can be partitioned into $(1-b)n$ plane subgraphs.

Motivated by these two results, Orthaber~\cite{workshopProblem1}
posed the following problem.

\begin{problem}[\cite{workshopProblem1}]
\label{prob-GG}
Is there a constant $c>0$ such that every set of $n$ points in the plane in general position contains a crossing family of size $cn$ or a non-crossing family of size $cn$?
\end{problem}

An affirmative solution to Problem~\ref{prob-GG} would imply that the edge set of any complete geometric graph on $n$ vertices can be partitioned into $cn$ plane subgraphs for some constant $c<1$.
This is a further motivation for studying Problem~\ref{problem-crVsNoncr}.

We note that a weaker variant of \cref{prob-GG} was independently posed by Schnider~\cite{workshopProblem2}; see Problem~\ref{prob-GG2} in Section~\ref{sec-openProblems}.
Bose, Hurtado, Rivera-Campo, and Wood~\cite{bhrw06} actually proved a stronger result, where the crossing family of size~$k$ is replaced by a more general configuration called a \emph{spoke set} by Schnider~\cite{workshopProblem2}.
A \emph{spoke set of size $k$} in a set $P$ of $n$ points in the plane in general position is a set $\mathcal{L}$ of $k$ pairwise non‐parallel lines such that in each unbounded region of the arrangement defined by the lines in $\mathcal{L}$ there lies at least one point of $P$; 
see \cref{fig-spokeSet}(a) for an example.
It is not difficult to see that if $P$ determines a crossing family of size $k$, then it determines a spoke set of size~$k$~\cite{bhrw06}.
The converse, however, is not true, as the example in \cref{fig-spokeSet}(b) illustrates. 

\begin{figure}[ht]
    \centering
    \includegraphics{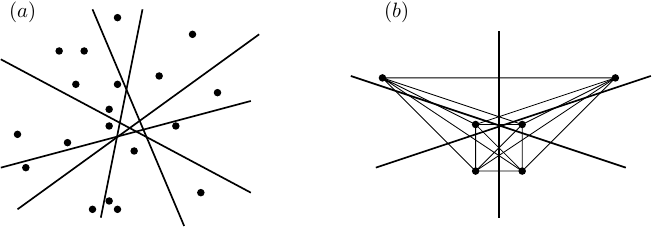}
    \caption{(a) A spoke set of size 5. (b) A set of 6 points in the plane in general position that determines a spoke set of size $3$ but does not determine a crossing family of size $3$.}
    \label{fig-spokeSet}
\end{figure}

Non-crossing families were also used by Dumitrescu and Pach~\cite{dumiPach24GD}, who proved that every complete geometric graph on a dense set of $n$ points can be partitioned into at most $cn$ plane subgraphs for some constant $c<1$.
Here, a set $P$ of $n$ points is \emph{dense} if the ratio between the maximum and the minimum distances in $P$ is of the order
of $\Theta(\sqrt{n})$.

\subsection{Motivation: geometric quasiplanarity of complete graphs}

For an integer $k\ge 2$, we say that a graph drawing is \emph{$k$-quasiplanar} if it contains no set of $k$ pairwise crossing edges. We say that a drawing of a graph $G$ is  \emph{straight-line} if each edge of $G$ is represented by a line segment between its endvertices. A graph $G$ is \emph{$k$-quasiplanar} if it admits a $k$-quasiplanar drawing. If $G$ admits a straight-line $k$-quasiplanar drawing, then $G$ is \emph{geometric $k$-quasiplanar}. 

The study of $k$-quasiplanar graphs was initiated in the 1990s and has gathered considerable attention from researchers in discrete geometry and graph drawing communities; see \cite{Ackerman09,Ackerman15,Angelini2018,AngeliniDLFS2024,AngeliniDLFS2024GD,Brandenburg2016,Capoyleas92,Felix23GD,Hoffmann}, among many others.
One of the problems in this line of research is determining the minimum value $k$ for which the complete graph $K_n$ is $k$-quasiplanar or geometric $k$-quasiplanar.
These values are only known for small values of $n$. In particular, Ackerman~\cite{Ackerman09} showed that $K_9$ is geometric $3$-quasiplanar
and Brandenburg \cite{Brandenburg2016} proved that $K_{10}$ is $3$-quasiplanar but not geometric $3$-quasiplanar. 
For larger values of $n$, it remains an interesting open problem to determine these values. In this language, Problem~\ref{problem-crVsNoncr} asks us to determine this value for specific classes of straight-line drawings of $K_n$, that is, the ones for which the underlying point set does not contain a non-crossing family of size $m$.

\subsection{Notation}

In this paper, we consider only finite sets of points in the plane in general position. For brevity, we sometimes just refer to them as point sets.
We use $\log$ to denote the logarithm with base 2.
For a positive integer $n$, we denote by $[n]$ the set $\{1,\dots,n\}$.
For the sake of clarity of presentation, we systematically omit floor and ceiling signs whenever they are not crucial. 
For a point set $P$, we use ${\rm conv}(P)$ to denote the convex hull of $P$.
If $A$ and $B$ are two sets of points in the plane, then we call them \emph{separated} if their convex hulls are disjoint.

\section{Our contribution}

We first show that forbidding non-crossing families of some small size can indeed help in finding larger crossing families than the ones obtained from Theorem~\ref{thm-natan}.
We actually find something stronger, a so-called \emph{convex bundle}.

For a positive integer $k$ and a point set $P$ in the plane, a \emph{convex bundle of size $k$} in $P$ is a $k$-tuple of disjoint non-empty subsets $C_1,\dots,C_k$ of $P$, 
such that for every choice of points $c_i \in C_i$ for every $i \in [k]$, the points $c_1,\dots,c_k$ are in convex position.
The \emph{width} of a convex bundle of size $k$ is the minimum over the cardinalities of its subsets $C_1,\dots,C_k$.

Our first result guarantees the existence of a convex bundle of size $\Theta(n/m)$ in any $n$-point set without a non-crossing family of size $m$.
To show it, we make use of the so-called Same-Type-Lemma~\cite{barVal98}, a classic result on orientations of point tuples in point sets. 

\begin{restatable}{theorem}{theoremNoncrVsCr}
\label{thm-noncrVsCr}
There is a constant $C > 0$ such that, for all positive integers $m$ and $k$, every set of at least $C k m$ points in the plane in general position contains either a convex bundle of size $k$ and width $m$ or a non-crossing family of size~$m$.
\end{restatable}

Since every convex bundle of size $2k$ and width at least $1$ determines a crossing family of size $k$, we obtain the following immediate corollary of Theorem~\ref{thm-noncrVsCr}.

\begin{corollary}
\label{cor-noncrVsCr}
There is a constant $C > 0$ such that, for all positive integers $m$ and $k$, every set of at least $C k m$ points in the plane in general position contains either a crossing family of size $k$ or a non-crossing family of size $m$. \qed
\end{corollary}

In particular, it follows from Corollary~\ref{cor-noncrVsCr} that if our point set does not contain a constant size 
non-crossing family, then it contains a crossing family of linear size.
Working with crossing families directly, we use Theorem~\ref{thm-noncrVsCr} in combination with \cref{thm-natanBipartite} below to improve the bound from \cref{cor-noncrVsCr} as follows, by this obtaining our main theorem.

\begin{restatable}{theorem}{theoremNoncrVsCrStronger}
\label{thm-noncrVsCrStronger}
There is a constant $C'>0$ such that, for every positive integer $m$, every set of $n \geq C'm$ points in the plane in general position contains either a non-crossing family of size $m$ or a crossing family of size $n/2^{O(\sqrt{\log{m}})}$.
\end{restatable}

Since no set of $n$ points in the plane contains a non-crossing family of size~$n$, Theorem~\ref{thm-natan} can be obtained as a corollary of Theorem~\ref{thm-noncrVsCrStronger} by setting $m=n$. The lower bound on the size of the crossing families that we obtain by Theorem~\ref{thm-noncrVsCrStronger} is asymptotically larger than the bounds from Theorem~\ref{thm-natan} and Corollary~\ref{cor-noncrVsCr},
as long as the point set does not contain a non-crossing family of size $n^{\Omega(1)}$. 
For example, if we forbid non-crossing families of size $\log{n}$, then we find a crossing family of  size $\Omega(n/2^{\sqrt{\log{\log{n}}}})$ instead of $n/2^{O(\sqrt{\log{n}})}$ and $\Omega(n/\log{n})$ that we would get from Theorem~\ref{thm-natan} and Corollary~\ref{cor-noncrVsCr}, respectively.
Further, Theorem \ref{thm-noncrVsCrStronger} implies that for any constant $m$, if $P$ is a set of $n$ points in the plane in general position with no non-crossing family of size $m$, then the complete geometric graph on $P$ can be decomposed into $cn$ plane subgraphs for some  constant $c<1$, by this resolving the problem of Bose, Hurtado, Rivera-Campo and Wood~\cite{bhrw06} for such point sets. 

To prove Theorem~\ref{thm-noncrVsCrStronger}, we use the following slight strengthening of Theorem~\ref{thm-natan}, which we obtain by modifying the proof of the result by Pach, Rubin, and Tardos~\cite{prt21}.

\begin{restatable}{theorem}{theoremNatanBipartite}
\label{thm-natanBipartite}
For every set $P$ of $n\ge2$ points in the plane in general position and every partition of $P$ into two separated subsets $P_1$ and $P_2$ such that $||P_1|-|P_2||\le 1$, $P$ contains a crossing family $\mathcal{F}$ of size at least $n/2^{O(\sqrt{\log{n}})}$ where each segment from $\mathcal{F}$ has one endpoint in $P_1$ and one endpoint in~$P_2$.
\end{restatable}

We also consider the algorithmic aspects of our results.
The proofs of Theorems~\ref{thm-noncrVsCr} and~\ref{thm-noncrVsCrStronger} are both constructive. By analyzing their time complexity, we obtain an expected 
linear time algorithm for Theorem~\ref{thm-noncrVsCr} and an expected polynomial time algorithm for Theorem~\ref{thm-noncrVsCrStronger}.  

\begin{restatable}{theorem}{theoremTimeComplexity}
\label{thm-timeComplexity}
There is a constant $C>0$ such that for all positive integers $k$ and~$m$, if $P$ is a set of $n=Ckm$ points in the plane in general position, 
then a convex bundle of size $k$ and width $m$ or a non-crossing family of size $m$ can be computed in expected time $O(n)$.

Further, a crossing family of size $n/2^{O(\sqrt{\log{m}})}$ or a non-crossing family of size $m$ can be computed in expected time $O(nm^{1+O((\log{m})^{-1/3})})$.
\end{restatable}

We remark that the non-determinism in the time complexities stems from an analysis of the runtime requirements of the Same-Type-Lemma. A deterministic variant of this building block would yield a deterministic analogue of \cref{thm-timeComplexity}.
Further, we note that the proof of Theorem~\ref{thm-natan} is also constructive and results in an $n^{2+O((\log{n})^{-1/3})}$ time algorithm for finding an according crossing family in a set of $n$ points; see~\cite{prt21}.

Finally, we show that the size of the largest spoke set and the largest crossing family in a point set can differ by a constant multiplicative factor, showcasing the different behavior of these two closely related notions. 

\begin{proposition} \label{prop-spokevscross}
For any odd integer $k$, there is a set $P$ of $3k-1$ points in the plane in general position, such that the largest spoke set in $P$ contains $\lfloor1.5k\rfloor$ lines but $P$ does not contain a crossing family of size larger than $k$. 
\end{proposition}

Note that if Problem \ref{prob-crossingfamily} admits an affirmative solution, then the sizes of the maximal spoke set and the maximal crossing family differ by at most a constant multiplicative factor. 

\section{Proofs of Theorems~\ref{thm-noncrVsCr} and~\ref{thm-noncrVsCrStronger}}\label{sec-main}

This section is devoted to proving our main result, the existence of a large crossing family or a large non-crossing family (Theorem~\ref{thm-noncrVsCrStronger}), 
for which we will use and first prove the existence of a large convex bundle or a large non-crossing family (Theorem~\ref{thm-noncrVsCr}).

One crucial ingredient for the proof of Theorem~\ref{thm-noncrVsCr} is a well-known result by B\'{a}r\'{a}ny and Valtr~\cite{barVal98}, called the \emph{Same-Type lemma}, on the combinatorics of point sets in~$\mathbb{R}^d$. 

For an integer $d \geq 2$ and non-coplanar points $p_1,\dots,p_{d+1} \in \mathbb{R}^d$, where $p_{i,j}$ is the $j$th coordinate of $p_i$, the \emph{orientation} of the $(d+1)$-tuple $(p_1,\dots,p_{d+1})$ is the sign of the determinant of the matrix
\[
\begin{pmatrix}
1 & 1 & \cdots & 1\\
p_{1,1} & p_{2,1} & \cdots & p_{d+1,1}\\
\vdots & \vdots & \vdots &\vdots \\
p_{1,d} & p_{2,d} &\cdots & p_{d+1,d}
\end{pmatrix}.
\]
For point sets $Y_1,\dots, Y_{d+1}$ in $\mathbb{R}^d$, we say that the sets $Y_1,\dots, Y_{d+1}$ have the \emph{same-type property} if, for every choice of points
$y_1 \in Y_1,\dots , y_{d+1} \in Y_{d+1}$, the orientation of the $(d+1)$-tuple $(y_1,\dots, y_{d+1})$ is the same.
More generally, for every integer $r \geq d+1$, we say that sets $Y_1,\dots, Y_r$ in $\mathbb{R}^d$ have the \emph{same-type property} if any $d + 1$ of them do.

With these definitions, we can now state the Same-Type lemma.

\begin{theorem}[Same-Type lemma~\cite{barVal98}]
\label{thm-sameTypeLemma}
For all integers $d \geq 2$ and $r \geq d+1$, there is a constant $c=c(d,r) >0$ such that for all pairwise disjoint sets $X_1,\dots, X_r$ in~$\mathbb{R}^d$ whose union is in general position,  there exist point sets $Y_1,\dots, Y_r$ having the same-type property and satisfying $Y_i \subseteq X_i$ and $|Y_i| \geq c|X_i|$.
\end{theorem}

B\'{a}r\'{a}ny and Valtr~\cite{barVal98} showed that $c(d,r) \geq (d+1)^{(-2^d-1)\binom{r-1}{d}}$.
This was later improved by Fox, Pach, and Suk~\cite{foxPachSuk16} to $c(d,r) \in 2^{-O(d^3r\log{r})}$.
Very recently, Bukh~\cite{Bukh2024} proved the bounds $d^{-50d^3}r^{-d^2} \leq c(d,r) \leq d^dr^{-d}$, which are polynomial in $r$.

For disjoint sets $P$ and $Q$ of points in the plane, we write $x(P)<x(Q)$ if the $x$-coordinate $x(p)$ of every point $p \in P$ is smaller than $x(q)$ for every point $q \in Q$.
For a positive integer $a$, we say that $a$ points in the plane form an \emph{$a$-cap} if they lie on the graph of some concave function.
Similarly, for $u \in \mathbb{N}$, a set of $u$ points in the plane forms a \emph{$u$-cup} if its points lie on the graph of some convex function. 
Note that any $a$-cap or $u$-cup is a point set in convex position.

\theoremNoncrVsCr* 

\begin{proof}
We set $C = 7 \cdot c(2,7)^{-1} \cdot c(2,5)^{-6}$, where $c(2,r) \in (0,1]$ is the constant from Theorem~\ref{thm-sameTypeLemma}. We assume $ k\geq 3$ as otherwise the statement is trivial.
Let $P$ be a set of $n= Ckm$ points in the plane in general position that does not contain a non-crossing family of size~$m$.
We partition $P$ by vertical lines into sets $P_1,\dots,P_7$ such that $x(P_1)<\cdots<x(P_7)$ and $|P_i| = n/7$ for every $i \in \{1,\dots,7\}$.
We now apply Theorem~\ref{thm-sameTypeLemma} to the 7-tuple $(P_1,\dots,P_7)$ and obtain the 7-tuple $\mathcal{B} = (P'_1,\dots,P'_7)$ 
that satisfies $P'_i \subseteq P_i$ and $|P'_i| \geq  c(2,7) \cdot n/7 \ge m$ for every $i \in \{1,\dots,7\}$.

We let $p_i$ be an arbitrary point from $P'_i$ for each $i \in \{1,\dots,7\}$.
We show that the set $S=\{p_1,\dots,p_7\}$ is in convex position.
Suppose for contradiction it is not.
Then, by Carath\'{e}odory's theorem, there is a 4-tuple $T$ of points from $S$ such that $T$ is not in convex position.
Since the collection $\mathcal{B}$ has the same-type property, the sets from $\mathcal{B}$ whose points are in $T$ form a non-crossing family of size $m$.
This contradicts our assumption that $P$ does not contain such a family.
Moreover, since $x(P'_1)<\dots<x(P'_7)$, it follows from Theorem~\ref{thm-sameTypeLemma} that there are sets $A$ and $U$ such that $A \cup U=\{1,\dots,7\}$, $A \cap U=\{1,7\}$, $\{p_i \colon i \in A\}$ is an $|A|$-cap, and $\{p_i \colon i \in U\}$ is a $|U|$-cup for every choice of $S=\{p_1,\dots,p_7\}$. 
By the pigeonhole principle, we have $|A| \geq 5$ or $|U| \geq 5$.
By symmetry, we can assume without loss of generality that $|A| \geq 5$ and we let $a_1<\dots<a_5$ be the 5 smallest elements of $A$.

Thus, we are done if $k \leq 5$, and we can assume $k \geq 6$.
We partition the set $P'_{a_3}$ by vertical strips into sets $P'_{a_3,1},\dots,P'_{a_3,k}$, each of size $|P'_{a_3}|/k$. 
Now, we let $Q_1 = P'_{a_1}$, $Q_2=P'_{a_2}$, $Q_{k-1} = P'_{a_4}$, $Q_k = P'_{a_5}$, and $Q_j = P'_{3,j}$ for every $j \in \{3,4,\dots,k-2\}$.
Observe that $x(Q_1)<\dots<x(Q_k)$ and $|Q_i| \geq c(2,7)\cdot n/(7k) \ge m$ for every $i \in [k]$.
Moreover, for every choice of $q_1 \in Q_1$, $q_2 \in Q_2$, $q_3 \in Q_3 \cup \cdots \cup Q_{k-2}$, $q_4 \in Q_{k-1}$, and $q_5 \in Q_k$, the points $q_1,\dots,q_5$ form a 5-cap.

Let $\pi$ be the permutation of $[k]$ defined by $\pi(i)=\lceil i/2 \rceil$ if $i$ is odd and $\pi(i)=k+1-i/2$ if $i$ is even.
Our goal is to show that for any $5 \leq \ell \leq k$, there exist large
sets  $R^\ell_i \subseteq Q_{\pi(i)}$ for $i\in[\ell]$ such that for all choices $r_i \in R^\ell_i$, the points $r_1,\dots,r_\ell$ form an $\ell$-cap; cf.~Figure~\ref{fig-convexBundle}. 
To this end, we first iteratively construct the sets $R^\ell_i$ such that any five sets $R^\ell_i, \dots, R^\ell_{i+4}$ fulfill the same-type property.  
Then, we reason about their sizes, and finally, we show by induction that they fulfill $\ell$-cap property.

\begin{figure}[ht]
    \centering
    \includegraphics{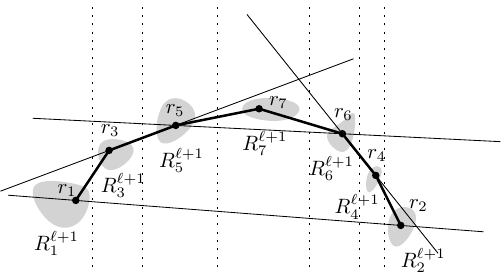}
    \caption{An illustration of the proof of Theorem~\ref{thm-noncrVsCr} for $\ell=6$.}
    \label{fig-convexBundle}
\end{figure}

For $\ell=5$, we simply set $R^\ell_i=Q_{\pi(i)}$ for each $i \in [\ell]$.
Note that these sets have the same-type property by construction, and that for all choices $r_i \in R^5_i$, the points $r_1,\dots,r_5$ form a $5$-cap.
We proceed by iteratively increasing $\ell$. 
Assume that we have sets $R^\ell_1,\dots,R^\ell_\ell$ for some integer $\ell$ with $5 \leq \ell < k$ such that $R^\ell_i \subseteq Q_{\pi(i)}$ for every $i \in [\ell]$ and such that any five sets $R^\ell_i, \dots, R^\ell_{i+4}$ fulfill the same-type property.  
We apply Theorem~\ref{thm-sameTypeLemma} to the 5-tuple $R^\ell_{\ell-3},R^\ell_{\ell-2},R^\ell_{\ell-1},R^\ell_{\ell},R^\ell_{\ell+1}=Q_{\pi(\ell+1)}$ ordered by $x$-coordinates.  
This way, we obtain sets $R^{\ell+1}_j \subseteq R^\ell_j$ for every $j \in \{\ell-3,\dots,\ell+1\}$ such that $|R^{\ell+1}_j| \geq c(2,5) \cdot |R^\ell_j|$ and the collection $R^{\ell+1}_{\ell-3},R^{\ell+1}_{\ell-2},R^{\ell+1}_{\ell-1},R^{\ell+1}_\ell,R^{\ell+1}_{\ell+1}$ has the same-type property.
For every $i \in [\ell-4]$, 
we set $R^{\ell+1}_i = R^\ell_i$.

Next, consider the sizes of the final sets $R^k_i$ for $i \in [k]$.
Note that we apply Theorem~\ref{thm-sameTypeLemma} at most five times within each set $Q_i$.
Thus, at the end we have, for every $i \in [k]$,
\[|R^{k}_i| \geq c(2,5)^5 \cdot |Q_{\pi(i)}| \geq c(2,5)^5 \cdot c(2,7)\cdot n/(7k) \geq m/c(2,5),\]
where we used the choice of $C$ and $n$.

Finally, we argue by induction on $\ell$ that 
for all choices $r_i \in R^{\ell}_i$, the points $r_1,\dots,r_{\ell}$ form an $\ell$-cap.
This is clearly true for $\ell = 5$ by the choice of $Q_1$, $Q_2$, $Q_3$, $Q_{k-1}$, and $Q_k$.
Thus, we assume $\ell \geq 5$ and proceed with the induction step $\ell \to \ell+1$.

Since $R^{\ell+1}_j \subseteq R^{\ell}_j$ for every $j \in [\ell]$, it follows from the induction hypothesis that the points $r_1,\dots,r_\ell$ form an $\ell$-cap for any choice of $r_i \in R^{\ell+1}_i$ with $i \in [\ell]$.
Thus, we only need to verify that $r_{\ell+1}$ extends this $\ell$-cap into an $(\ell+1)$-cap.

Since $x(R^{\ell+1}_{\ell-1})<x(R^{\ell+1}_{\ell+1})<x(R^{\ell+1}_{\ell})$ or $x(R^{\ell+1}_{\ell})<x(R^{\ell+1}_{\ell+1})<x(R^{\ell+1}_{\ell-1})$, all points from $R^{\ell+1}_{\ell+1}$ lie in a vertical strip $S$ between $R^{\ell+1}_{\ell-1}$ and $R^{\ell+1}_\ell$.
Since for every choice of $q_1 \in Q_1$, $q_2 \in Q_2$, $q_3 \in Q_3 \cup \cdots \cup Q_{k-2}$, $q_4 \in Q_{k-1}$, and $q_5 \in Q_k$, the points $q_1,\dots,q_5$ form a 5-cap, all points from $\cup_{i \in \{3,4,\dots,\ell+1\}}R^{\ell+1}_i$ lie above any line determined by one point from $R^{\ell+1}_1$ and one point from $R^{\ell+1}_2$.

Since the collection $R^{\ell+1}_{\ell-3},R^{\ell+1}_{\ell-2},R^{\ell+1}_{\ell-1},R^{\ell+1}_\ell,R^{\ell+1}_{\ell+1}$ has the same-type property, for every choice of one point from each set of this collection, the resulting 5-tuple is in convex position. 
Otherwise, similarly as before, it follows from Carath\'{e}odory's theorem that $P$ contains a non-crossing family of size $m/c(2,5) \geq m$, which is impossible by assumption. 
It follows from this fact that all points from $R^{\ell+1}_{\ell+1}$ lie below every line determined by a point from $R^{\ell+1}_{\ell-3}$ and a point from $R^{\ell+1}_{\ell-1}$ (otherwise there are four points from $R^{\ell+1}_{\ell-3}$, $R^{\ell+1}_{\ell-1}$, $R^{\ell+1}_{\ell+1}$, and $R^{\ell+1}_{\ell}$ that are not in convex position).
Similarly, all points from $R^{\ell+1}_{\ell+1}$ lie below every line determined by a point from $R^{\ell+1}_{\ell-2}$ and a point from $R^{\ell+1}_\ell$.
All points from $R^{\ell+1}_{\ell+1}$ lie on the same side of all lines determined by a point from $R^{\ell+1}_{\ell-1}$ and a point from $R^{\ell+1}_{\ell}$.
If they lie above all such lines, then $r_1,\dots,r_{\ell+1}$ form an $(\ell+1)$-cap, since $r_{\ell+1}\in S$ is contained in the triangle spanned by lines $\overline{r_{\ell-3}r_{\ell-1}}$, $\overline{r_{\ell-2}r_{\ell}}$, and $\overline{r_{\ell-1}r_{\ell}}$, and we are done.
Thus, we suppose that all points from $R^{\ell+1}_{\ell+1}$ lie below these lines.

Then, for every choice $r_1 \in R^{\ell+1}_1$, $r_2 \in R^{\ell+1}_2$, $r_{\ell-1} \in R^{\ell+1}_{\ell-1}$, $r_\ell \in R^{\ell+1}_\ell$, and $r_{\ell+1} \in R^{\ell+1}_{\ell+1}$, the points $r_1,r_2,r_{\ell-1},r_\ell,r_{\ell+1}$ are not in convex position.
This is because $r_{\ell+1}$ is located in the part of the vertical strip $S$ that lies between the two line segments $r_1r_2$ and $r_{\ell-1}r_\ell$.
We apply Theorem~\ref{thm-sameTypeLemma} to the 5-tuple $(R^{\ell+1}_{1},R^{\ell+1}_{2},R^{\ell+1}_{\ell-1},R^{\ell+1}_{\ell},R^{\ell+1}_{\ell+1})$ and obtain sets $S_j \subseteq R^{\ell+1}_j$ for every $j \in \{1,2,\ell-1,\ell,\ell+1\}$ such that $|S_j| \geq c(2,5) \cdot |R^{\ell+1}_j|$
and $(S_1,S_2,S_{\ell-1},S_\ell,S_{\ell+1})$ has the same-type property. 
By selecting a point from each $S_j$ we cannot get a point set in convex position, since no 5-tuple of points from $R^{\ell+1}_{1},R^{\ell+1}_{2},R^{\ell+1}_{\ell-1},R^{\ell+1}_{\ell},R^{\ell+1}_{\ell+1}$ is in convex position.
It follows from the same-type property and from Carath\'{e}odory's theorem that some 4-tuple of sets $S_j$ forms a non-crossing family.
This family has size $|S_j| \geq c(2,5) \cdot |R^{\ell+1}_j| \geq m$, which contradicts the fact that $P$ does not contain such a non-crossing family.

Altogether, the sets $R^k_1,\dots,R^k_k$ form a convex bundle of size $k$ and width at least $m$.
\end{proof}

To obtain the bound in Theorem~\ref{thm-noncrVsCrStronger}, 
we combine Theorem~\ref{thm-noncrVsCr} with Theorem~\ref{thm-natanBipartite} (the proof of which can be found in \cref{sec-partition}).

\theoremNoncrVsCrStronger*

\begin{proof} 
Let $C'=6C$, where $C$ is the constant from \cref{thm-noncrVsCr}, 
and let $P$ be a set of $n \geq C'm$ points in general position that does not contain a non-crossing family of size $m$.
Further, let $k= \lfloor n/(2Cm) \rfloor \geq 3$, and note that $2kCm \le n < 2(k+1)Cm \le  \frac{8}{3}kCm$. 
It follows from \cref{thm-noncrVsCr} applied with $2k$ that $P$ contains a convex bundle of size $2k$ and width $m$. We denote the sets which form this convex bundle by $A_1,A_2,\dots, A_k, B_1, B_2,\dots,B_k$ in cyclical order given by the convex position.  

For every $i \in [k]$, we apply Theorem~\ref{thm-natanBipartite} to the set $A_i \cup B_i$ and obtain a crossing family $\mathcal{F}_i$ of size at least $m/2^{c\sqrt{\log{m}}}$ for some constant $c>0$ such that each of its line segments has one endpoint in $A_i$ and the other in $B_i$, where the constant $c$ does not depend on the cardinality $m$ of the point sets. 

By the definition of the convex bundle, each line segment from $\mathcal{F}_i$ crosses each line segment from $\mathcal{F}_j$ for all $i$ and $j$ with $1 \leq i < j \leq k$.
It follows that the line segments in $\cup_{i=1}^k \mathcal{F}_i$ form a crossing family of size
	\[\frac{km}{2^{c\sqrt{\log m}}} \ge  \frac{3}{8C}\cdot \frac{n}{2^{c\sqrt{\log m}}} \in \frac{n}{2^{O(\sqrt{\log m})}},\]
which completes the proof.
\end{proof}

\section{Proof of Theorem~\ref{thm-natanBipartite}}\label{sec-partition}

In this section, we show that every set of $n$ points in the plane in general position partitioned into two separated nonempty subsets $P_1$ and $P_2$ of equal size 
contains a crossing family $\mathcal{F}$ of size at least $n/2^{O(\sqrt{\log{n}})}$, where each segment from $\mathcal{F}$ has one endpoint in~$P_1$ and one endpoint in $P_2$.
We do so by slightly modifying the proof of Theorem~\ref{thm-natan} by Pach, Rubin, and Tardos~\cite{prt21}.

Let $\mathcal{L}$ be a finite set of lines in the plane.
Following the terminology in~\cite{prt21}, we call the partition $\mathcal{A} = \mathcal{A}(\mathcal{L})$ of $\mathbb{R}^2 \setminus \left(\bigcup \mathcal{L}\right)$ into 2-dimensional sets, called \emph{cells}, the \emph{arrangement}\footnote{We remark that in other literature, the line arrangement formed by a set $\mathcal{L}$ of lines also contains $\mathcal{L}$, partitioned into vertices, which are the crossings of $\mathcal{L}$, and edges, which are the (possibly unbounded)
segments of lines of $\mathcal{L}$ between consecutive crossings.} of lines from $\mathcal{L}$. 
Each cell of $\mathcal{A}$ is a maximal connected
region of $\mathbb{R}^2 \setminus \left(\bigcup \mathcal{L}\right)$ and a (possibly unbounded) convex polygon.
The boundary of each cell consists of \emph{edges}, which are portions of the lines from $\mathcal{L}$.
The edges connect crossings between lines from $\mathcal{L}$, which are called \emph{vertices}.
The \emph{zone} of a line $\ell \not\in \mathcal{L}$ in the arrangement $\mathcal{A}$ is the union of all the open cells whose closure is intersected by $\ell$.

We will need the following result on point sets and line arrangements, which was implicitly established by Matou\v{s}ek~\cite{mat02,Matousek91socg}; see~\cite{prt21} for a proof.

\begin{lemma}[\cite{mat02}]
\label{lem-cuttingLemma}
For any $n$-element point set $P$ in the plane in general position and for any $\varepsilon > 0$, there exists a set $\mathcal{L}$ of 
$O(1/\varepsilon^2)$ lines such that the zone of any 
	line $\ell \notin \mathcal{L}$ within 
	the arrangement $\mathcal{A}(\mathcal{L})$  contains at most $\varepsilon n$ points of $P$. 
\end{lemma}

The following definitions were introduced by Pach, Rubin, and Tardos~\cite{prt21}.
For non-empty separated point sets $A$ and $B$, we define a binary relation $<_B$ on $A$ such that, for distinct points $x,y \in A$, we let $x <_B y$ if $B$ is contained in the open half-plane to the left of the line determined by $x$ and $y$ and oriented from $x$ towards $y$.
Pach, Rubin, and Tardos~\cite{prt21} proved that  $<_B$ is a partial order on~$A$, in which two distinct points $x$ and $y$ are incomparable if and only if the line through $x$ and $y$ intersects ${\rm conv}(B)$.

For a partially ordered set $(S,<_S)$, we use $\iota(S, <_S)$ to denote the number of pairs of elements from $S$ that are incomparable in $<_S$.
For $\varepsilon>0$, two separated point sets $A$ and $B$, each with $m$ points,
are said to form an \emph{$\varepsilon$-avoiding} pair if 
$\iota(A, <_B) + \iota(B, <_A) \leq \varepsilon m^2.$

The following lemma is a variant of Lemma 3.3 by Pach, Rubin, and Tardos~\cite{prt21}. Its proof is a modification of their argument.

\begin{lemma}
\label{lem-lem3.3}
There is an absolute constant $c > 0$ with the following property. 
For every positive integer $m$ and every real number $\varepsilon \in (0,1)$, if $P$ is a set of $n$ points in the plane in general position such that $n \geq \frac{cm}{\varepsilon^4}$ and $P$ is partitioned into two separated subsets $P_1$ and $P_2$ with $||P_1|-|P_2||\leq 1$, 
then there are two (separated) 
$m$-element sets $A \subseteq P_1$ and $B \subseteq P_2$ that form an $\varepsilon$-avoiding pair.
\end{lemma}
\begin{proof}
We first apply Lemma~\ref{lem-cuttingLemma} to obtain a set $\mathcal{L}$ of $r=c_0/\varepsilon^2$ lines for some constant $c_0>0$ and an arrangement $\mathcal{A} = \mathcal{A}(\mathcal{L})$ such that the zone of any other line within $\mathcal{A}$ contains at most $\varepsilon n/16$ points of $P$.
We assume that $r \geq 3$ and choose a constant $c$ such that $c \geq 16c^2_0$.

We split each cell of $\mathcal{A}$ with parallel segments or half-lines that do not pass through any point of $P$ into smaller cells such that all but at most one of them (which we will call \emph{exceptional}) contain precisely
$m$ points of $P$, and if there is an exceptional cell, it contains fewer than $m$ points.
Since there is at most one exceptional (smaller) cell per cell of $\mathcal{A}$,  there are at most $\frac{r^2+r+1}{2} < r^2$ exceptional cells in total.
Let $\Pi$ be the resulting cell decomposition of $\mathbb{R}^2$ induced by $\mathcal{A}$ and the additional segments and half-lines.
Note that every cell in $\Pi$ is convex.
We call the set of $m$ points of $P$ inside a non-exceptional cell of $\Pi$ a \emph{cluster}.
Let $D$ be the number of clusters and note that $D \leq n/m$.

Let $\ell$ be a line that separates $P_1$ from $P_2$.
Such a line exists as $P_1$ and $P_2$ are separated.
By the choice of $\mathcal{A}$, the zone of $\ell$ in $\mathcal{A}$ contains at most $\varepsilon n/16$ points of $P$.
These points lie in at most $\varepsilon n/(16m)$ clusters, as each such cluster contains $m$ points from the zone of $\ell$ in~$\mathcal{A}$.

There are at most $r^2m$ points of $P$ that do not belong to any cluster as the number of exceptional (smaller) cells is at most $r^2$ and each such a cell contains less than $m$ points of~$P$.
Thus, since $|P_1|=n/2$, the number of clusters that contain points from $P_1$ is at least $(n/2-r^2 m)/m$, and an analogous claim is true for $P_2$.

The number of pairs $(C_1,C_2)$ of clusters $C_1$ and $C_2$ such that $C_1  \subseteq P_1$ and $C_2   \subseteq P_2$ is at least 
\begin{align*}
\left(\frac{n/2 - r^2m}{m}\right)^2 - \frac{\varepsilon n}{16m}D &\geq \frac{(n/2-c_0^2m/\varepsilon^4)^2}{m^2} - \frac{\varepsilon n^2}{16m^2} \\
&\geq   \frac{(n/2-n/16)^2}{m^2} - \frac{n^2}{16m^2} > \frac{n^2}{8m^2},
\end{align*}
as every cluster containing points from $P_1$ and $P_2$ is in the zone of $\ell$ and the number of exceptional cells is at most $r^2$.
We also used the choice of $n$, $r=c_0/\varepsilon^2$, $\varepsilon < 1$, $D \leq n/m$, $n \geq cm/\varepsilon^4$, and $c \geq 16 c_0^2$.

Let 
\[X = \sum_{(C_1,C_2)}(\iota(C_1,<_{C_2}) + \iota(C_2,<_{C_1})),\]
where the sum is taken over all ordered pairs $(C_1,C_2)$ of distinct clusters such that $C_1 \subseteq P_1$ and $C_2  \subseteq P_2$.

We now estimate $X$ according to the point pairs that are incomparable with respect to a cluster.
An unordered pair of distinct points $\{x, y\} \in \binom{P}{2}$ contributes to $X$ only if $x$ and
$y$ come from the same cluster whose points from $P$ are contained in $P_1$ or $P_2$.
In this case, the contribution of $\{x, y\}$ is the number of other clusters whose points from $P$ are contained in $P_2$ or~$P_1$, respectively, 
and whose convex hulls are crossed by the line $\ell'$ containing $x$ and $y$. 
All of these clusters belong
to the zone of $\ell'$ in the arrangement $\mathcal{A}$. 
By the choice of $\mathcal{A}$, the zone of any
line contains at most $\varepsilon n/16$ vertices. Thus, the contribution of any pair of points is at
most $\varepsilon n/(16m)$ and we obtain
\[
X \leq D \binom{m}{2} \frac{\varepsilon n}{16m} < \frac{\varepsilon n^2}{32},
\]
since $D \leq n/m$.

On the other hand, every pair of clusters $(C_1,C_2)$ which is not $\varepsilon$-avoiding contributes more than $\varepsilon m^2$ to $X$.
Therefore, there are at most $X/(\varepsilon m^2)$ such pairs,
implying that the number of not $\varepsilon$-avoiding pairs of clusters is less than
\[\frac{\varepsilon n^2}{32} \cdot \frac{1}{\varepsilon m^2} = \frac{n^2}{32m^2}.\]
Clearly, each cluster has $m$ points of $P$ and any two of them are separated. 

Putting everything together, the number of pairs $(C_1,C_2)$ of clusters such that $C_1  \subseteq P_1$ and $C_2  \subseteq P_2$ is at least $\frac{n^2}{8m^2}$, which is larger than the number $\frac{n^2}{32m^2}$ of pairs of clusters that are not $\varepsilon$-avoiding.
Thus, there is a pair of clusters $(A,B)$ that is $\varepsilon$-avoiding and satisfies $A \subseteq P_1$ and $B  \subseteq P_2$.
\end{proof}

Finally, we state the following lemma by Pach, Rubin, and Tardos~\cite{prt21}.

\begin{lemma}[\cite{prt21}]
\label{lem-lem3.7}
Let $s$ be a positive integer and set $K = 8^{\binom{s}{2}}$, $M = 9^s K$, $\varepsilon = 2^{-3s-11}$.
Assume that $A$ and $B$ are $M$-element point sets that form an $\varepsilon$-avoiding pair and $A\cup B$ is in general position.
Then, we can find $K$ pairwise crossing segments, each connecting a point of $A$ to a point of $B$.
\end{lemma}

We are now ready to prove Theorem~\ref{thm-natanBipartite} by slightly modifying the proof of Theorem~\ref{thm-natan} by Pach, 
Rubin, and Tardos~\cite{prt21}.

\theoremNatanBipartite*

\begin{proof} 
Assuming $n \geq 3$, let $s$ be the smallest positive integer such that $P$ does not determine a crossing family of size $K=8^{\binom{s}{2}}$ that has all line segments between $P_1$ and $P_2$.
We have $s\geq 1$, since $P$ determines a crossing family of size $1=8^{\binom{1}{2}}$.

By Lemma~\ref{lem-lem3.7}, there is no pair $\{A,B\}$ of sets, with $A$ containing $M=9^sK$ points of $P_1$ and $B$ containing $M$ points of $P_2$, that is $\varepsilon= 2^{-3s-11}$-avoiding.
Applying Lemma~\ref{lem-lem3.3} to $P$ gives $n \in O(M\varepsilon^{-4}) =2^{O(s)}K$.

By the choice of $s$, the point set $P$ determines a crossing family of size $K' = 8^{\binom{s-1}{2}} \in K/2^{O(s)}$  with all line segments between $P_1$ and $P_2$.
Therefore, we have $n>K' \in K/2^{O(s)}$ and $s=O(\sqrt{\log {n}})$.
We also have $K' \in K/2^{O(s)} = n/2^{O(s)} = n/2^{O(\sqrt{\log{n}})}$.
\end{proof}

\section{Proof of Theorem~\ref{thm-timeComplexity}}\label{sec-timeComplexity}
This section is devoted to the algorithmic aspects of~\cref{thm-noncrVsCr,thm-noncrVsCrStronger}.
Specifically, we prove that there is a constant $C>0$ such that for all positive integers $k$ and $m$ if $P$ is a set of $n=Ckm$ points in the plane in general position, we can find a convex bundle of size $k$ and width $m$ or a non-crossing family of size $m$ in time $O(n)$.
Similarly, we show that we can find a crossing family of size $n/2^{O(\sqrt{\log{m}})}$ or a non-crossing family of size $m$ in time $O(nm^{1+O((\log{m})^{-1/3})})$.

To do so, we use a variant of the semi-algebraic regularity lemma recently proved by Rubin~\cite{rubin24}.
To state it, we first need to introduce some more definitions.

A Boolean function $\psi \colon \mathbb{R}^{d \times r} \to \{0, 1\}$ is an \emph{$r$-wise semi-algebraic relation} in~$\mathbb{R}^d$ if it can be described by a finite combination $(f_1,\dots, f_s, \phi)$, where $f_1,\dots, f_s \in \mathbb{R}[z_1,\dots, z_{rd}]$ are polynomials, and $\phi$ is a Boolean function in $\{0, 1\}^s \to \{0, 1\}$, so that 
\[\psi(y_1,\dots, y_r) =
\phi(f_1(y_1,\dots, y_r) \leq 0; \dots ; f_s(y_1, \dots, y_r) \leq 0)\] holds for all the ordered $r$-tuples $(y_1,\dots, y_r)$ of points $y_i \in \mathbb{R}^d$.
We call the $(s+1)$-tuple $(f_1,\dots , f_s, \phi)$ a \emph{semi-algebraic description of $\psi$} (which need not be unique).
Such a relation $\psi$ has \emph{description complexity at most $(\Delta, s)$} if it admits a description using at most $s$ polynomials $f_i \in \mathbb{R}[z_1,\dots, z_{rd}]$ of maximum degree at most~$\Delta$.
We let $\Psi_{d,r,\Delta,s}$ be the family of all such $r$-wise semi-algebraic relations $\psi \colon \mathbb{R}^{d \times r} \to \{0, 1\}$ with description complexity bounded by~$(\Delta, s)$.

\begin{theorem}[\cite{rubin24}]
\label{thm-semialgebraicRegularityPartite}
The following statement holds for any fixed integers $d \geq 1$, $r \geq 2$, $\Delta \geq 0$, $s \geq 1$, and any fixed $\delta > 0$.
Let $X_1\dots, X_r$ be sets of $n$ points in $\mathbb{R}^d$, and $\psi$ be an $r$-wise relation in $\Psi_{d,r,\Delta,s}$ which is satisfied for at least $\varepsilon |X_1| \cdot \cdots \cdot |X_r|$ of the $r$-tuples $(x_1,\dots, x_r) \in X_1 \times \cdots \times X_r$.
Then one can find, in expected time $O(\sum_{i=1}^r (|X_i| + 1/\varepsilon) \log(1/\varepsilon))$, subsets $Y_i \subseteq X_i$, each of cardinality $\Omega(\varepsilon^{d+1+\delta}|X_i|)$, so that $\psi(y_1,\dots, y_r) = 1$ holds for all $(y_1,\dots, y_r) \in Y_1 \times \cdots \times Y_r$.
\end{theorem}

The proof of the Same-Type Lemma by Fox, Pach, and Suk~\cite{foxPachSuk16} is constructive. We go through it step by step in order to estimate the running time of the algorithm it yields.
If $(p_1,\dots,p_n)$ is a sequence of $n \geq d+1$ points in $\mathbb{R}^d$ in general position, then the \emph{order-type} of $(p_1,\dots, p_n)$ is the mapping $\chi \colon \binom{P}{d+1} \to \{+1, -1\}$, which assigns to 
each $(d + 1)$-tuple of these points its orientation (positive or negative).
For integers $d \geq 2$ and $r \geq d+1$, let $X_1,\dots, X_r$ be pairwise disjoint sets in~$\mathbb{R}^d$, each containing $n$ points, whose union is in general position.
It is known that the number of different
order-types of $r$-element point sets in $\mathbb{R}^d$ is at most $r^{O(d^2r)}$~\cite{goodmanPollack91,goodmanPollack93}.
Thus, by the pigeonhole principle, at least
$r^{-O(d^2r)}|X_1|\cdots|X_r|$ $r$-tuples $(x_1,\dots, x_r) \in (X_1,\dots, X_r)$ have the same order-type $\chi$.
We define the relation $E \subset X_1 \times \dots \times X_r$, where $(x_1,\dots, x_r) \in E$ if and only if $(x_1,\dots, x_r)$ has the order-type $\chi$.
Then it can be shown~\cite{foxPachSuk16} that the description complexity of $E$ is $\left(\binom{r}{d+1},1\right)$.
Therefore, we can apply Theorem~\ref{thm-semialgebraicRegularityPartite} to the $r$-partite semi-algebraic hypergraph $(P, E)$ with $d$, $r$, $\Delta = \binom{r}{d+1}$, $s=1$, $\delta = 1$, and $\varepsilon = r^{-O(d^2 r)}$ 
to obtain subsets $Y_i \subseteq X_i$, each of cardinality $c(d,r) \cdot |X_i|$, with the same-type property, where $c(d,r) \in r^{-O(d^3 r)}$, in time 
\[
O\left(r\left(n+r^{O(d^2r)}\right) \log{\left(r^{O(d^2r)}\right)}\right) \leq O((n+r^{O(d^2r)})d^2r^2\log{r}).
\]
Thus, we obtain the following estimate on the time complexity of the Same-Type Lemma.

\begin{corollary}
\label{cor-sameTypeTime}
For all integers $d \geq 2$ and $r \geq d+1$, the sets $Y_1,\dots,Y_r$ from Theorem~\ref{thm-sameTypeLemma} 
can be found in expected time $O((n+r^{O(d^2r)})d^2r\log{r})$ if $|X_1|=\dots=|X_r|=n$.\qed
\end{corollary}

\theoremTimeComplexity*

\begin{proof} 
Let $P$ be a set of at least $\Theta(k m)$ points in the plane in general position.
In the proof of Theorem~\ref{thm-noncrVsCr}, we apply the Same-Type Lemma (Theorem~\ref{thm-sameTypeLemma}) with $d=2$, $r=6$, $n=\Theta(m)$ $O(k)$-times to obtain a convex bundle $A_1,\dots,A_k,B_1,\dots,B_k$ of size $2k$ and width $m$ in $P$.
Since $k = \Theta(|P|/m)$, it follows from Corollary~\ref{cor-sameTypeTime} applied with $d=2$, $r=6$, and $n=|P|$ that this can be done in time
\[O\left(\frac{|P|}{m} \cdot m\right) = O(|P|),\]
which is the estimate on the time complexity in Theorem~\ref{thm-noncrVsCr}.

To prove Theorem~\ref{thm-noncrVsCrStronger}, we also apply Theorem~\ref{thm-natanBipartite} to each set $A_i \cup B_i$ with $1 \leq i \leq k$.
The proof of Theorem~\ref{thm-natanBipartite} is an adaptation of the proof of Theorem~\ref{thm-natan} and has asymptotically the same time complexity.
As mentioned by Pach, Rubin, and Tardos~\cite{prt21}, their proof yields an algorithm whose running time is $n^{2+O((\log{n})^{-1/3})}$ on sets of $n$ points.
Since $k = \Theta(|P|/m)$ and $|A_i \cup B_i| \in O(m)$, it follows from this and the previous estimate that the total running time is
\[O(|P| + (|P|/m) \cdot m^{2+O((\log{m})^{-1/3})}) = O(|P| \cdot m^{1+O((\log{m})^{-1/3})}),\]
which completes the proof.
\end{proof}

\section{Discussion and open problems}\label{sec-openProblems}

As the main result in this work, we have obtained a Ramsey-type result that gives a lower bound on the size of a maximum crossing family in a set $P$ of $n$ points in general position in the plane in dependence of the size $m$ of a maximum non-crossing family in $P$. The resulting bound of $n/2^{O(\sqrt{\log{m}})}$ for the crossing family improves the currently best known lower bound of $n/2^{O(\sqrt{\log{n}})}$ whenever $m\in o(n)$. Our proofs are constructive and yield an algorithm that finds a crossing family of desired size in  polynomial expected time. 

The result 
also constitutes a relevant step towards answering \cref{prob-GG}. Unfortunately, the obtained bounds are not strong enough to completely resolve Problem~\ref{prob-GG}, which seems difficult.
Thus, the first natural open problem is to improve the currently best lower bound $k \in n/2^{O(\sqrt{\log{m}})}$ on the size $k$ of the largest crossing family determined by a set of $n$ points in the plane in general position that does not contain a non-crossing family of size $m$.
In particular, it is interesting to find out for how large values of $m$ we can get that $k \in \Omega(n)$. 
It follows from both Corollary~\ref{cor-noncrVsCr} and Theorem~\ref{thm-noncrVsCrStronger} that if $m \in O(1)$, then $k \in \Omega(n)$.
Is it possible to obtain a linear lower bound on $k$ even if $m \in \omega(1)$?

\begin{problem}
\label{prob-linear}
Does there exist a constant $c>0$ and a function $f \colon \mathbb{N} \to \mathbb{N}$ with $\lim_{n \to \infty} f(n)=\infty$ such that every set of $n$ points in the plane in general position contains either a crossing family of size at least $cn$ or non-crossing family of size at least $f(n)$?
\end{problem}

Bose, Hurtado, Rivera-Campo, and Wood showed that
the edge set of every complete geometric graph with a spoke set of size $k$ can be partitioned into $n - k$ plane trees~\cite{bhrw06}. 
Motivated by this result, Schnider~\cite{workshopProblem2} posed the following variant of Problem~\ref{prob-GG}.

\begin{problem}[\cite{workshopProblem2}]
\label{prob-GG2}
Is there a constant $c>0$ such that every set of $n$ points in the plane in general position contains a spoke set of size $cn$ or a non-crossing family of size $cn$?
\end{problem}

Since every crossing family is a spoke set but not every spoke set is a crossing family, Problem~\ref{prob-GG2} is a weaker variant of Problem~\ref{prob-GG}.
However, already an affirmative solution of Problem~\ref{prob-GG2} would imply that the edge set of any complete geometric graph can be partitioned into $cn$ plane subgraphs for some constant $c<1$.

\medskip
Finally, the well-known \cref{prob-crossingfamily}, which asks whether or not $T(n) \in \Theta(n)$~\cite{braMosPach05}, still remains open.

\bibliography{bibliographyGD}

\appendix

\section{Proof of Proposition~\ref{prop-spokevscross}}

    Let $S$ be a regular $k$-gon with vertices $v_1,\dots,v_k$. We first replace every vertex $v_i$ of $S$ by a segment $s_{v_i}$ of length $2\epsilon$ (where $\epsilon>0$ is constant) centered at $v_i$ and perpendicular to the bisector of the angle of $S$ at $v_i$. 
    Then, we consider $S'$ to be the set of $2k$ points formed by the endpoints of $s_v$. Note that the innermost cell of the complete geometric graph on $S'$ is again a regular $k$-gon. We call this cell $C$. 
    Now, let $\mathcal{L} = \{\ell_1,\ell_2,\dots, \ell_k\}$ be the set of lines where each $\ell_i$ is the extension of the angle-bisectors of $S$ at $v_i$. Clearly, $\mathcal{L}$ is a spoke set of size $k$ on $S'$. Now, for each odd $i\in [\frac{k}{2}-1]$, add to $\mathcal{L}$ the line bisecting the angle between $\ell_i$ and $\ell_{i+\lceil\frac{k}{2}\rceil}$, call this new set $\mathcal{L}'$. We finally obtain $P$ by adding to each empty unbounded cell $E$ of $\mathcal{A}(\mathcal{L})$ a single point inside $E\cap C$; see Figure \ref{fig-spokesets}. 
    
    Now, by construction, $|P|=3k-1$ and $\mathcal{L}'$ is a spoke set of size $\lfloor 1.5k \rfloor$ on $P$. Also, no line segment between two points in $P$ which lie outside $C$ crosses any of the segments between two points in $P \cap C$. Now, let $\mathcal{F}$ be a crossing family in $P$.  If there is a segment $s\in \mathcal{F}$ which has both endpoints in $P\setminus C$ then the line through $s$ divides $P$ into two parts, one of which contains at most $k-1$ points of $P$. Therefore $|\mathcal{F}|\le k$, so we can assume this is not the case. Otherwise, each segment of $\mathcal{F}$ has one endpoint in $P\cap C$, but $|P\cap C| =k$, so again $|\mathcal{F}|\le k$.

\begin{figure}[htbp]
    \centering
    \begin{subfigure}[b]{0.4\textwidth}
        \includegraphics[width=\textwidth,page=1]{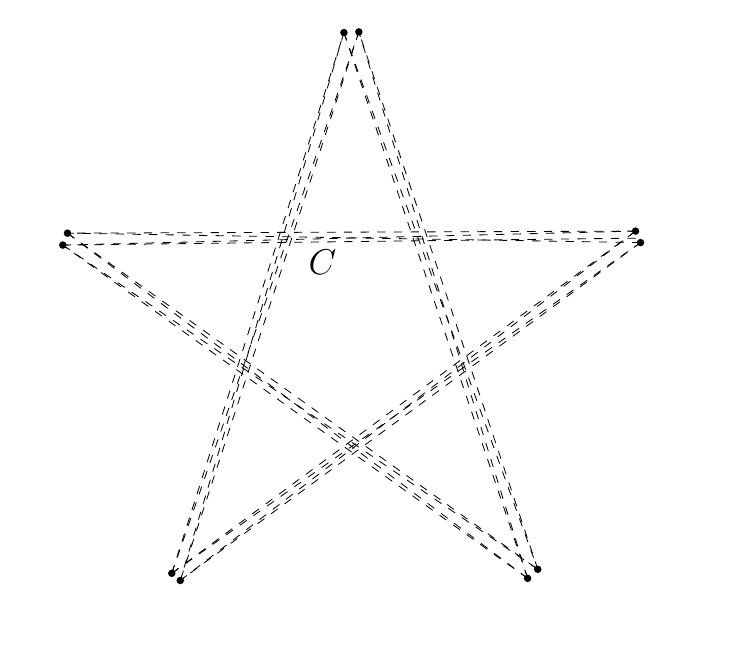}
        
    \end{subfigure}
    \hfill
    \begin{subfigure}[b]{0.4\textwidth}
        \includegraphics[width=\textwidth,page=2]{spokesetsconst.pdf}
      
    \end{subfigure}
    \hfill
    \begin{subfigure}[b]{0.4\textwidth}
        \includegraphics[width=\textwidth, page=3]{spokesetsconst.pdf}
       
    \end{subfigure}
    \caption{Construction of the set $P$ from Proposition \ref{prop-spokevscross}.}
    \label{fig-spokesets}
\end{figure}

\end{document}